\documentclass[a4paper,10pt]{article}

\usepackage{amsfonts}
\usepackage{amsmath}
\usepackage[numbers, square]{natbib}
\usepackage{graphicx}
\usepackage{color}

\newenvironment{proof}[1][Proof]{\begin{trivlist}
\item[\hskip \labelsep {\bfseries #1}]}{\end{trivlist}}
\newenvironment{definition}[1][Definition]{\begin{trivlist}
\item[\hskip \labelsep {\bfseries #1}]}{\end{trivlist}}

\newtheorem{theorem}{Theorem}[section]
\newtheorem{lemma}[theorem]{Lemma}
\newtheorem{proposition}[theorem]{Proposition}

\newcommand{\qed}{\nobreak \ifvmode \relax \else
      \ifdim\lastskip<1.5em \hskip-\lastskip
      \hskip1.5em plus0em minus0.5em \fi \nobreak
      \vrule height0.75em width0.5em depth0.25em\fi}
\newcommand{\lp}{\left(}
\newcommand{\rp}{\right)}
\newcommand{\si}{\sigma}
\newcommand{\De}{\Delta}

\title{Knots with distinct primitive/primitive and primitive/Seifert representatives}
\author{Brandy Guntel}

\begin{document}
\bibliographystyle{plainnat}

\maketitle {}

\begin{abstract}
Berge introduced knots that are primitive/primitive with respect to the genus 2 Heegaard surface, $F$, in $S^3$; surgery on such knots at the surface slope yields a lens space. Later Dean described a similar class of knots that are primitive/Seifert with respect to $F$; surgery on these knots at the surface slope yields a Seifert fibered space. Here we construct a two-parameter family of knots that have distinct primitive/Seifert embeddings in $F$ with the same surface slope, as well as a family of torus knots that have a primitive/primitive representative and a primitive/Seifert representative with the same surface slope.
\end{abstract}

\section{Introduction}

Since every closed 3-manifold can be obtained from Dehn surgery on a link $L$ in
$S^3$, much effort is dedicated to understanding Dehn surgery on knots and links. Every knot can be embedded in a genus $g$ Heegaard surface in $S^3$. For example, the torus knots can be embedded in a genus 1 Heegaard surface in $S^3$.  In this paper, we will focus on knots that can be
embedded in the genus 2 Heegaard surface in $S^3$. In \citep{hill} and \citep{himu}, Hill and
Murasugi studied such knots, which they called double-torus knots. Two subclasses of the double-torus knots are the primitive/primitive knots and the primitive/Seifert knots, which arise in the study of exceptional Dehn surgery.

A theorem of Thurston tells us that at most finitely many surgeries on a
hyperbolic knot are non-hyperbolic. Since these non-hyperbolic surgeries are
uncommon, we refer to them as exceptional surgeries. In \citep{berge}, Berge
introduced the primitive/primitive knots and observed that they have lens space
surgeries. Later Dean (\citep{deansthesis}, \citep{dean}) introduced the primitive/Seifert
knots, a natural generalization of primitive/primitive knots, and noted that
surgery on such a knot at the surface slope is either a Seifert fibered space or a connected sum of
lens spaces. In \citep{mimo}, Miyazaki and Motegi showed that the
primitive/Seifert knots are mostly hyperbolic. Since Seifert fibered surgeries on hyperbolic knots are difficult to understand, the primitive/Seifert knots are particularly interesting to study.

A natural question is that of uniqueness: can a knot have more than one primitive/Seifert representative with the same surface slope? Here, in section \ref{pt1}, we give examples of a two-parameter infinite family of knots with distinct primitive/Seifert embeddings. We also ask a similar question: can a knot have two representatives, one primitive/Seifert and one primitive/primitive, with the same surface slope? In section \ref{pt2}, we discuss a family of torus knots that have this property. All the necessary definitions are found in section \ref{prelim}.

The author would like to thank her thesis advisor, Cameron Gordon, for many valuable conversations and suggestions, as well as his patience and encouragement. The author would also like to thank John Berge for helpful suggestions. This work is partially supported by NSF RTG Grant DMS-0636643.

\section{Preliminaries}\label{prelim}

\subsection{Primitive and Seifert knots}

We begin by letting $K$ be a simple closed curve in the genus 2 surface $F$,
which bounds a genus 2 handlebody $H$, and we can consider the space obtained
by adding a 2-handle, homeomorphic to $D^2\times I$, to $H$ along $K$. This is
done by identifying $\partial D^2 \times I$ with an annulus neighborhood of $K$
in $F$, a process called 2-handle addition. The 2-handle addition may result in
several types of spaces. We name two of them here. 

\begin{definition}
$K$ is called \textit{primitive with respect to $H$} if adding a 2-handle to $H$
 along $K$ yields a solid torus.
\end{definition}

\begin{definition}
$K$ is called \textit{Seifert with respect to $H$} if adding a 2-handle to $H$
along $K$ yields a Seifert fibered space.
\end{definition}

If $K$ is Seifert, a lemma of Dean \citep{dean} and Eudave-Mu\~{n}oz \citep{em} 
tells us which Seifert fibered spaces occur.

\begin{lemma} \label{2HA}
If $K$ is Seifert with respect to $H$, then the manifold obtained by adding a
2-handle to $H$ along $K$ is Seifert fibered over the disk with at most two
exceptional fibers with multiplicities $a_1$ and $a_2$ or over the M\"obius band
with at most one exceptional fiber of multiplicity $b$. In the first case, $K$
is primitive if and only if $a_1$ or $a_2$ is 1.
\end{lemma}

In this paper, we will only consider knots that are Seifert over the disk. If
$K$ is Seifert with respect to  $H$ over the disk with exceptional fibers of
multiplicity $a_1$ and $a_2$, we say $K$ is \textit{$\lp a_1, a_2 \rp$ Seifert
fibered over $D^2$} or simply \textit{$(a_1, a_2)$ Seifert}. 

Now consider $H$ to be a genus 2 handlebody in the Heegaard decomposition of
$S^3$ and call the other handlebody $H'$.  Then $F = \partial H = \partial H'$
and $K$ is a simple closed curve in the genus 2 Heegaard surface $F$ of $S^3$.
Primitive and Seifert with respect to $H'$ are defined in the same way as for $H$, so
we can define primitive/Seifert as follows.

\begin{definition}
The curve $K$ is called \textit{primitive/Seifert with respect to $F$} if it is
primitive with respect to $H$ and Seifert, but not primitive, with respect to $H'$.
\end{definition}

Up to now, we have considered curves on the genus 2 Heegaard surface $F$ of
$S^3$, but we can also think of these curves as knots in $S^3$. Since 2-handle
addition along $K$ on $H$ and $H'$ yields either two solid tori or a Seifert
fibered space and a solid torus, we can describe the manifolds obtained by surgery
on the knot at the surface slope with respect to $F$ \citep{dean} \citep{em}, defined here.

\begin{definition}
Let $N$ denote a tubular neighborhood of $K$ in $S^3$. The \textit{surface slope of
$K$ with respect to $F$} is the isotopy class of $\partial N \cap F$ in
$\partial N$.
\end{definition}

\begin{proposition}\label{primSeifsurgtypes}
If a knot $K$ in $S^3$ is primitive/Seifert with respect to the genus 2 Heegaard
surface $F$, then Dehn surgery on $K$ at the surface slope yields one of the following:
\begin{enumerate}
\item[a.] A Seifert fibered space of the form $S^2\lp a_1, a_2, a_3 \rp$
\item[b.] A Seifert fibered space of the form $\mathbb{RP}^2 \lp b_1, b_2 \rp$
\item[c.] A connected sum of two lens spaces
\end{enumerate}
\end{proposition}

\subsection{Twisted Torus Knots}

Let $T \lp p, q \rp $ denote the $\lp p, q \rp$-torus knot. We obtain from $T\lp
p,q \rp$ the \textit{twisted torus knot} $K\lp p, q, r, n \rp$ by twisting $r$
strands of $T\lp p,q \rp$ $n$ times. This new knot can be viewed as a curve on a
genus 2 Heegaard surface in $S^3$, which one can see in the following way.

\begin{figure}[h]
 \begin{center}
\includegraphics[scale=.3]{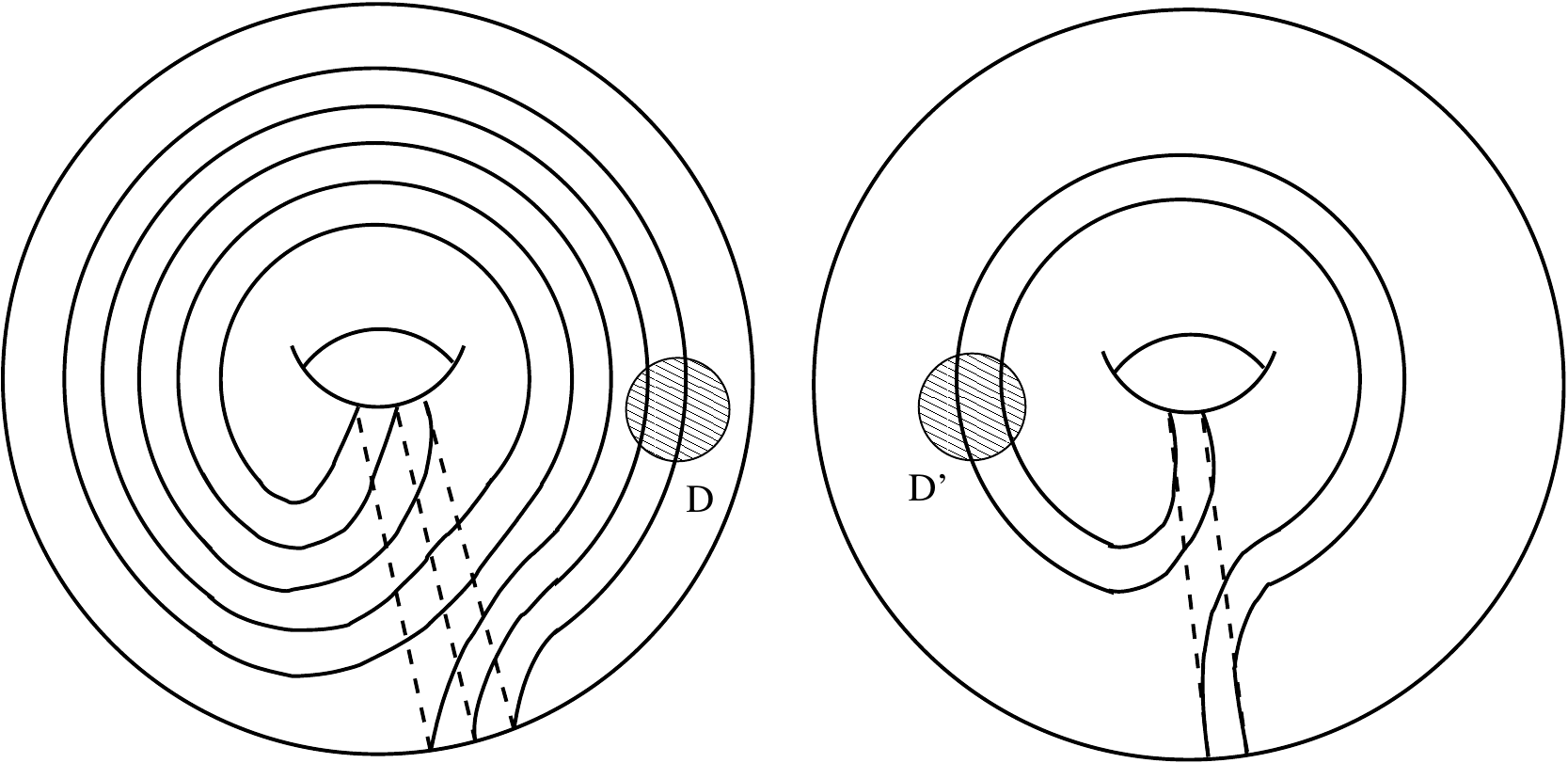} 
\caption{$K\lp 5, -3, 2, -1\rp$}\label{fig:ttk}
\end{center}
\end{figure}

Let $D$ be a disk on the torus so that $T\lp p,q \rp$ intersects $D$ in $r$
disjoint arcs, where $0\le r\le p+q$. We also consider a disk $D'$ on the torus
in which $r$ parallel copies of $T \lp 1, n \rp$, denoted $rT\lp1,n\rp$, lie so
that $D'$ intersects $rT\lp 1, n \rp$ in $r$ disjoint arcs, one in each
component of $rT\lp 1, n\rp$. Then we excise the disks $D$ and $D'$ from their
respective tori and glue the two tori together along the boundary of the disks
so that the orientations of the two torus links align correctly. Figure
\ref{fig:ttk} shows the example $K\lp 5, -3, 2, -1\rp$. 

Note that if $r=0$, the twisted torus knot $K\lp p, q, 0, n\rp$ lies on the
torus so that the disk $D$ is disjoint from the $\lp p, q \rp$-torus knot, so
that the knot only intersects one of the two punctured tori whose union is $F$. 

\begin{proposition}[Dean]\label{slope}
 The surface slope of $K\lp p,q, r, n \rp$ is $pq+nr^2$.
\end{proposition}

From Proposition \ref{primSeifsurgtypes}, when a twisted torus knot is
primitive/Seifert, $pq+nr^2$ surgery will result in one of the types of
manifolds listed there. It remains to discuss which twisted torus knots are
primitive/Seifert.

\subsection{Primitive and Seifert Twisted Torus Knots}

By considering the word for $K$ in $\pi_1 \lp H \rp$ and the algebraic
definitions of primitive and Seifert, Dean showed the following \citep{dean}.

\begin{proposition}\label{SFS}
For any integer $k$ with $1\le k < \frac{p}{q}$, $K\lp p, q, p-kq, n\rp$ is $\lp
k, p-kq\rp$ Seifert fibered over $D^2$.
\end{proposition}

\begin{theorem}\label{primwhen}
$K$ is primitive with respect to $H$ if and only if one of the following
conditions holds:
\begin{enumerate}
 \item[a.] $p=1$
 \item[b.] $r \equiv \pm1$ or $\pm q$ mod $p$. 
\end{enumerate}
\end{theorem}

We can modify these statements slightly when we want to consider $H'$ rather
than $H$. For example, the $\lp p,q\rp$-torus knot part of the twisted torus
knot is a $\lp q,p \rp$-torus knot on the boundary of $H'$. Then the requirement
$p=1$ in Theorem \ref{primwhen} becomes $q=1$ when we replace $H$ with $H'$ in
that theorem.

These statements lead to criteria on the parameters $p$, $q$, $r$ and $n$ for
twisted torus knots to be primitive/Seifert. By finding a regular fiber of the
Seifert fibered space that results from 2-handle addition along $K$ and using
homological arguments, Dean found the Seifert fibered space that results from
$pq+nr^2$ surgery on $K$ \citep{dean}.

\begin{theorem}\label{psttks}
 The twisted torus knots $K\lp p, q, r, n \rp$, with $r < \max \lbrace p,q
\rbrace$ and $n = \pm 1$, that are Seifert with respect to $H$ and primitive
with respect to $H'$ are $K\lp p,q,p-kq, n\rp$ with $1<q<\frac{p}{2}$ and $2\le
k\le \frac{p-2}{q}$. Furthermore, surgery at the surface slope for these knots
yield Seifert fibered spaces of the form $S^2 \lp k, p-kq, p- \lp k-n\rp q \rp$.

\end{theorem}

From this theorem, we see that the curves $K\lp 17, 5, 2, -1 \rp$ and $K\lp 18,
5, 3, -1 \rp$ have the same surface slope with respect to the genus 2 Heegaard
surface in $S^3$, namely 81. They also yield the same Seifert fibered space after
surgery at the surface slope: $S^2 \lp 2, 3, 5 \rp$. (It is enough here to give
the multiplicities of the exceptional fibers because we know the surgery slope, and
hence the homology of the space.) On the other hand, Proposition \ref{SFS} 
shows that $K\lp 17, 5, 2, -1 \rp $ is $\lp 2, 5\rp$ Seifert fibered over $D^2$
with respect to $H$, whereas $K\lp 18, 5, 3, -1 \rp$ is $\lp 3, 5\rp$ Seifert
fibered over $D^2$ with respect to $H$. This tells us that there is no
homeomorphism of $S^3$ that preserves the Heegaard splitting and sends one curve
to the other. It is natural to ask whether these curves are the same as knots
in $S^3$. In this case, they are isotopic in $S^3$, which we can see using
conjugacy of elements in the braid group. This example provides insight for a
more general statement.

\section{Distinct Primitive/Seifert Embeddings}\label{pt1}

\begin{theorem}\label{T1}
Let the curves $K_1$ and $K_2$ in the genus 2 Heegaard surface $F$ be the
twisted torus knots $K\lp kq +\frac{q-1}{2},q, \frac{q-1}{2}, -1\rp$ and $K\lp
kq+\frac{q+1}{2},q, \frac{q+1}{2}, -1 \rp$, respectively, where $q\ge 5 $ is odd
and $k\ge 2$. Then $K_1$ and $K_2$ are isotopic as knots in $S^3$ and have the same surface slope with respect to $F$, but there is
no homeomorphism of $S^3$ sending the pair $\lp F, K_1\rp$ to $\lp F, K_2\rp$.
\end{theorem}

Here we consider $r$ to be $\frac{q-1}{2}$, so that $\frac{q+1}{2}$ is $r+1$, and
$p$ to be $kq +\frac{q-1}{2}$, so that $ kq+\frac{q+1}{2}$ is $p+1$. In
particular, we can express the parameters of the family either in terms of $q$
and $k$ or in terms of $p$ and $r$, with $p$ and $r$ being dependent on $q$ and
$k$. For notational ease, we will mostly use $p$ and $r$.

We think of the knots as closures of braids with $q$ strands, so we will carry out
calculations in the braid group $B_q$. We know that the $\left( p,q \right)
$-torus knot can be represented as a braid by $\left( \sigma_{q-1} \sigma_{q-2}
\cdots \sigma_1 \right) ^p$. Since the twisted torus knots are obtained from the
torus knot by twisting $r$ strands $n$ times, we can represent the twisted torus
knot $K \left( p, q, r, -1 \right) $ by $\left( \sigma_{q-1} \sigma_{q-2} \cdots
\sigma_1 \right) ^p \left(\sigma_1^{-1} \sigma_2^{-1}\cdots
\sigma_{r-1}^{-1}\right) ^r$ in $B_q$. If we can find an element $b\in B_q$ for
which $\left( \sigma_{q-1} \sigma_{q-2} \cdots \sigma_1 \right) ^p
\left(\sigma_1^{-1} \sigma_2^{-1}\cdots \sigma_{r-1}^{-1}\right) ^r$ and $b^{-1}
\left( \sigma_{q-1} \sigma_{q-2} \cdots \sigma_1 \right) ^{p+1}
\left(\sigma_1^{-1} \sigma_2^{-1}\cdots \sigma_{r}^{-1}\right) ^{r+1} b$ are
equal, then the knots $K \left(kq+ \frac{q-1}{2}, q, \frac{q-1}{2}, -1  \right)$
and $K \left(  kq+ \frac{q+1}{2}, q, \frac{q+1}{2}, -1 \right)$ are isotopic in
$S^3$. (Note: $\left( \sigma_{q-1} \sigma_{q-2} \cdots \sigma_1 \right) ^p
\left(\sigma_1^{-1} \sigma_2^{-1}\cdots \sigma_{r-1}^{-1}\right) ^r$ actually
represents $K\lp q,p, r,-1\rp$. In his thesis \citep{deansthesis}, Dean showed
that if $r<p$ and $r<q$, then $K\lp q,p, r,-1\rp$ and $K\lp p,q, r,-1\rp$ are
isotopic.)

Since $p=kq+\frac{q-1}{2}$, we can write this braid as \[ \left( \sigma_{q-1}
\sigma_{q-2} \cdots \sigma_1 \right) ^{kq+\frac{q-1}{2}} \left(\sigma_1^{-1}
\sigma_2^{-1}\cdots \sigma_{r-1}^{-1}\right) ^r. \] From braid theory
\citep{braidbook}, $\left( \sigma_{q-1} \sigma_{q-2} \cdots \sigma_1 \right) ^q
= \left( \sigma_1 \sigma_2 \cdots \sigma_{q-1} \right) ^q$. Since the right hand
side of this equality generates the center of the braid group $B_q$, the two
braids in question are conjugate when $\left( \sigma_{q-1} \sigma_{q-2} \cdots
\sigma_1 \right) ^{\frac{q-1}{2}} \left(\sigma_1^{-1} \sigma_2^{-1}\cdots
\sigma_{r-1}^{-1}\right) ^r$ and $\left( \sigma_{q-1} \sigma_{q-2} \cdots
\sigma_1 \right) ^{\frac{q+1}{2}} \left(\sigma_1^{-1} \sigma_2^{-1}\cdots
\sigma_{r}^{-1}\right) ^{r+1}$ are conjugate. We know that $r=\frac{q-1}{2}$, so
we can rewrite these braids as $\beta_1 = \left( \sigma_{2r} \sigma_{2r-1}
\cdots \sigma_1 \right) ^{r} \left(\sigma_1^{-1} \sigma_2^{-1}\cdots
\sigma_{r-1}^{-1}\right) ^r$ and $\beta_2 = \left( \sigma_{2r} \sigma_{2r-1}
\cdots \sigma_1 \right) ^{r+1} \left(\sigma_1^{-1} \sigma_2^{-1}\cdots
\sigma_{r}^{-1}\right) ^{r+1}$. Note that these braids are independent of $k$.

\begin{proposition}\label{P1}
$\beta_1$ and $\beta_2$ are conjugate by \[ \left( \sigma_1 \right) \left(
\sigma_2 \sigma_1 \right) \cdots \left( \sigma_{r-1} \sigma_{r-2} \cdots
\sigma_1 \right) \left( \sigma_{r+1} \right) \left( \sigma_{r+2} \sigma _{r+1}
\right) \cdots \left( \sigma_{2r} \sigma_{2r-1} \cdots \sigma_{r+1} \right). \]
\end{proposition}
 
 We will adopt the notation of Garside \citep{gars}, with some modification:
$\Pi_s^l = \sigma_l \sigma_{l+1} \cdots \sigma_s$ and $\Delta_s^l = \Pi_s^l
\Pi_{s-1}^l \cdots \Pi_l^l$. When $l=1$, we will leave off the superscript.
(Note: because the superscript may be confused with an exponent, any exponents
will occur outside of parentheses.) As in \citep{gars}, $\mbox{rev} w$ denotes the
word obtained by writing the elements of $w$ in the reverse order. This
notation makes the conjugating element in Proposition \ref{P1} easy to write as
$\mbox{rev}\Delta_{r-1} \mbox{rev} \Delta_{2r}^{r+1}$, and the braids $\beta_1$ and $\beta_2$
are also simplified to $ \left( \mbox{rev} \Pi_{2r} \right)^r  \left( \mbox{rev}
\Pi_{r-1} \right)^{-r}$ and  $ \left( \mbox{rev} \Pi_{2r} \right)^{r+1}  \left(
\mbox{rev} \Pi_{r} \right)^{-r-1}$.

Using this notation, we will prove several lemmas about braids that will be
helpful to prove Proposition \ref{P1}.

\begin{lemma}\label{L1}
For $l <  t \leq s$,  $\si_t  \Pi_s^l = \Pi_s^l \si_{t-1}$.
\end{lemma}

\begin{proof}
\begin{eqnarray*}
\si_t \Pi_s^l & = & \si_t \si_l \si_{l+1} \cdots \si_s \\
                      & = &  \si_l \cdots \si_{t-2} \si_t \si_{t-1} \si_t
\si_{t+1} \cdots \si_s \\
                      & = &  \si_l \cdots \si_{t-2} \si_{t-1} \si_t \si_{t-1}
\si_{t+1} \cdots \si_s \\
                      & = &  \si_l \cdots \si_s \si_{t-1} \\
                      & = &  \Pi_s^l \si_{t-1}
\end{eqnarray*}
\begin{flushright}
$\square$
\end{flushright}
\end{proof}

The following lemma is proved in the same manner.

\begin{lemma}\label{L3}
For $l <  t \leq s$,  $\si_{t-1} ( \Pi_s^l ) ^{-1} = ( \Pi_s^l ) ^{-1} \si_{t}$.
\end{lemma}

These two lemmas are general statements that we can move generators of $B_q$
past $\Pi_s^l$ at the expense of changing the generator to the previous or the
next generator. The following four lemmas are more specialized to fit into the
proof of Theorem \ref{T1}.

\begin{lemma}\label{L5}
$\De_{r-1} (\Pi_{r-1})^{-r} = (\De_{r-1})^{-1}$
\end{lemma}

\begin{proof} The proof employs repeated use of Lemma \ref{L3}.
\begin{eqnarray*}
\De_{r-1} (\Pi_{r-1})^{-r} & = & \Pi_{r-1} \Pi_{r-2} \cdots \Pi_1
(\Pi_{r-1})^{-r} \\
                                           & = & \Pi_{r-1} \Pi_{r-2} \cdots
\Pi_2 (\Pi_{r-1})^{-r+2} \Pi_{r-1}^{r-1} (\Pi_{r-1})^{-2} \\
                                           & = &  \Pi_{r-1} \Pi_{r-2} \cdots
\Pi_2 (\Pi_{r-1})^{-r+2} (\Pi_{r-2})^{-1} (\Pi_{r-1})^{-1} \\
                                           & = &  \Pi_{r-1} \Pi_{r-2} \cdots
\Pi_3  (\Pi_{r-1})^{-r+3} \Pi_{r-1}^{r-2} (\Pi_{r-1})^{-1} (\Pi_{r-2})^{-1}
(\Pi_{r-1})^{-1} \\ 
                                           & =& \Pi_{r-1} \Pi_{r-2} \cdots \Pi_3
 (\Pi_{r-1})^{-r+3} (\Pi_{r-3})^{-1} (\Pi_{r-2})^{-1} (\Pi_{r-1})^{-1}\\
                                           & = & \Pi_{r-1} \cdots \Pi_j
(\Pi_{r-1})^{-r+j} (\Pi_{r-j})^{-1} (\Pi_{r-j+1})^{-1}  \cdots (\Pi _{r-1})^{-1}
\\
                                           & = & \Pi_{r-1} \Pi_{r-2} (\Pi_{r-1})
^{-2} (\Pi_{2})^{-1} (\Pi_{3})^{-1} \cdots (\Pi_{r-1})^{-1} \\
                                           & = &  \Pi_{r-1} (\Pi_{r-1}) ^{-1} 
(\Pi_{1})^{-1} (\Pi_{2})^{-1} \cdots (\Pi_{r-1})^{-1} \\
                                           & = & (\Pi_{1})^{-1} (\Pi_{2})^{-1}
\cdots (\Pi_{r-1})^{-1} \\
                                           & = & ( \De_{r-1})^{-1} \\
\end{eqnarray*}
\begin{flushright}
$\square$
\end{flushright}
\end{proof}

\begin{lemma}\label{L6}
$\De_r^2 (\Pi_r )^{-r} = (\De_r)^{-1}$
\end{lemma}

The proof of Lemma \ref{L6} is very similar to that of Lemma \ref{L5}

\begin{lemma}\label{L7}
$( \De_r ) ^{-1} (\Pi_{2r})^{r+1} = \Pi_{2r}^{r+1} \Pi_{2r}^r \cdots \Pi_{2r}^2
\Pi_{2r}$
\end{lemma}

\begin{proof}
\begin{eqnarray*}
( \De_r ) ^{-1} (\Pi_{2r})^{r+1} & = & (\Pi_1)^{-1} (\Pi_2)^{-1} \cdots
(\Pi_r)^{-1}  (\Pi_{2r})^{r+1}\\
                                & = & (\Pi_1)^{-1} (\Pi_2)^{-1} \cdots
(\Pi_{r-1})^{-1} \Pi_{2r}^{r+1} (\Pi_{2r})^{r} \\
                                & = & \Pi_{2r}^{r+1} (\Pi_1)^{-1} (\Pi_2)^{-1}
\cdots (\Pi_{r-1})^{-1} (\Pi_{2r})^{r} \\
                                & = & \Pi_{2r}^{r+1} (\Pi_1)^{-1} (\Pi_2)^{-1}
\cdots (\Pi_{r-2})^{-1} \Pi_{2r}^{r}  (\Pi_{2r})^{r} \\
                                & = & \Pi_{2r}^{r+1} \Pi_{2r}^{r} \cdots
\Pi_{2r}^{j+2} (\Pi_1)^{-1} (\Pi_2)^{-1} \cdots (\Pi_{j})^{-1} (\Pi_{2r})^{j+1}
\\
                                & = & \Pi_{2r}^{r+1} \Pi_{2r}^{r} \cdots
\Pi_{2r}^{4} (\Pi_1)^{-1} (\Pi_2)^{-1} (\Pi_{2r})^{3} \\
                                & = & \Pi_{2r}^{r+1} \Pi_{2r}^{r} \cdots
\Pi_{2r}^{3}  (\Pi_1)^{-1} (\Pi_{2r})^{2} \\
                                & = & \Pi_{2r}^{r+1} \Pi_{2r}^{r} \cdots
\Pi_{2r}^{2} \Pi_{2r} \\
\end{eqnarray*}
\begin{flushright}
$\square$
\end{flushright}
\end{proof}

\begin{lemma}\label{L8}
$(\Pi_{2r} )^{r+1} (\De_{2r}^{r+1} )^{-1} = \Pi_{2r} \Pi_{2r-1} \cdots \Pi_r$
\end{lemma}

Lemma \ref{L8} is proved similarly to Lemma \ref{L7}.

Now we can easily prove Proposition \ref{P1}. 

\begin{proof}[Proof of Proposition \ref{P1}]
Using the modified Garside notation, we can rewrite the statement of the
proposition. Then the proposition will be true if and only if
\small{$\left(\mbox{rev}\Pi_{2r}\right)^r \left(\mbox{rev} \Pi_{r-1}
\right)^{-r}\mbox{rev}\De_{r-1} \mbox{rev}\De_{2r}^{r+1} =\mbox{rev}\De_{r-1}
\mbox{rev}\De_{2r}^{r+1}\lp\mbox{rev} \Pi_{2r} \rp^{r+1} \lp\mbox{rev} \Pi_r
\rp^{-r-1}$} in $B_q$. It is easy to see that two words are equal if and only if
their reverse words are equal. Using this, we can change the original equation
to: 

\begin{equation}\label{eq0}
\De_{2r}^{r+1} \De_{r-1} \left( \Pi_{r-1} \right)^{-r} \left(\Pi_{2r}\right)^r =
\lp \Pi_r \rp^{-r-1} \lp \Pi_{2r} \rp^{r+1} \De_{2r}^{r+1} \De_{r-1}
\end{equation} 

Next use the fact that $\De_{r-1} \left( \Pi_{r-1} \right)^{-r} $ and
$\De_{2r}^{r+1}$ commute and Lemma \ref{L1} to say that the proposition is true
if and only if 

\begin{equation}\label{eq1}
\De_{r-1} \left( \Pi_{r-1} \right)^{-r} \left(\Pi_{2r}\right)^r \De_r = \lp
\Pi_r \rp^{-r-1} \lp \Pi_{2r} \rp^{r+1} \De_{2r}^{r+1} \De_{r-1}
\end{equation} 

Now multiply on the right by $\lp \De_{r-1} \rp^{-1}$, so we have

\begin{equation}\label{eq2}
\De_{r-1} \left( \Pi_{r-1} \right)^{-r} \left(\Pi_{2r}\right)^r \Pi_r = \lp
\Pi_r \rp^{-r-1} \lp \Pi_{2r} \rp^{r+1} \De_{2r}^{r+1}
\end{equation}

Use Lemma \ref{L5} to rewrite this equation as 
\begin{equation}\label{eq3}
(\De_{r-1} )^{-1} (\Pi_{2r} )^r \Pi_r= ( \Pi_r ) ^ {-r-1} (\Pi_{2r})^{r+1}
\De_{2r}^{r+1}
\end{equation}

Next multiply by $\De_{r-1}$ on the left and use Lemmas \ref{L1} and \ref{L6} to
obtain
\begin{eqnarray*}
(\Pi_{2r} )^r \Pi_r & = & \De_{r-1} ( \Pi_r ) ^ {-r-1} (\Pi_{2r})^{r+1}
\De_{2r}^{r+1} \\
& = & (\Pi_r)^{-1} \De_{r}^2 ( \Pi_r ) ^ {-r} (\Pi_{2r})^{r+1} \De_{2r}^{r+1} \\
& = & (\Pi_r)^{-1} ( \De_r )^{-1} (\Pi_{2r})^{r+1} \De_{2r}^{r+1} \\
\end{eqnarray*}
 
Multiply on the left by $\Pi_r$, with the result  
\begin{equation}\label{eq6}
\Pi_r (\Pi_{2r} )^r \Pi_r  = (\De_r)^{-1} (\Pi_{2r})^{r+1} \De_{2r}^{r+1}  
\end{equation}

Lemma \ref{L1} gives that the left side of this equation is $\Pi_r
\Pi_{2r}^{r+1} (\Pi_{2r} )^r$. Since $\Pi_r \Pi_{2r}^{r+1} = \Pi_{2r}$,  
\begin{equation}\label{eq7}
(\Pi_{2r} )^{r+1} =  ( \De_r ) ^{-1} (\Pi_{2r} )^{r+1} \De_{2r}^{r+1}
\end{equation}

Finally multiply on the right by $(\De_{2r}^{r+1} ) ^{-1}$ to obtain the
equation 
\begin{equation}\label{eq8}
 (\Pi_{2r} )^{r+1} ( \De_{2r}^{r+1} ) ^{-1}= ( \De_r ) ^{-1} (\Pi_{2r} )^{r+1} 
\end{equation}

By Lemmas \ref{L1}, \ref{L7} and \ref{L8}, equation \eqref{eq8} becomes 
\begin{eqnarray*}
\Pi_{2r} \Pi_{2r-1} \cdots \Pi_r & = &  \Pi_{2r}^{r+1} \Pi_{2r}^r \cdots
\Pi_{2r}^2 \Pi_{2r} \\
                                                     & = &  \Pi_{2r} \Pi_{2r-1}
\cdots \Pi_r                       
\end{eqnarray*}
\begin{flushright}
$\square$
\end{flushright}
\end{proof}

Hence we have Proposition \ref{P1}, which we use to prove Theorem \ref{T1}.

\begin{proof}[Proof of Theorem \ref{T1}]
It is easy to see that these parameters meet the requirements $1<q<\frac{p}{2}$
and $2\le k\le \frac{p-2}{q}$. By Proposition \ref{slope}, the curves $K_1 =
K\lp kq +\frac{q-1}{2},q, \frac{q-1}{2}, -1 \rp$ and $K_2 = K\lp 
kq+\frac{q+1}{2},q, \frac{q+1}{2}, -1 \rp$ both have surface slope $ \lp
kq+\frac{q-1}{2} \rp q - \lp \frac{q-1}{2}\rp^2$. Proposition \ref{SFS} tells us
that $H[K_1] \cong D^2\lp k, \frac{q-1}{2} \rp$ and $H[K_2] \cong D^2\lp k,
\frac{q+1}{2} \rp$. Because $K_1$ and $K_2$ are both Seifert with respect to
$H$ and primitive with respect to $H'$, a homeomorphism $h$ of $S^3$ with $h\lp F, K_1\rp =\lp F, K_2 \rp$ would have to send $H$ to itself. Then $h$ must extend to a
homeomorphism of the two Seifert fibered spaces, $H[K_1]$ and $H[K_2]$, which is impossible. Hence no
such homeomorphism exists. We have shown that the two braids $\beta_1$ and
$\beta_2$ are conjugate, so their closures, $K_1$ and $K_2$, are isotopic in
$S^3$. 
\begin{flushright}
$\square$
\end{flushright}
\end{proof}

\begin{section}{Primitive/primitive and primitive/Seifert representatives}\label{pt2}

In this section, we describe a family of torus knots that have a primitive/primitive and a primitive/Seifert representative with the same surface slope. The primitive/Seifert representatives of this family of torus knots arises from the twisted torus knot construction, so this family of knots can be viewed as a family of twisted torus knots that are actually torus knots. 
From \citep{moser}, we know that the only integral surgery slopes on the torus knot $T(p,q)$ that yield lens spaces are $pq \pm 1$. Surgery at the surface slope on a primitive/primitive representative of a knot yields a lens space \citep{berge}. Hence, if a $(p,q)$ torus knot is to have a primitive/primitive representative, it must lie on the genus 2 surface with surface slope $pq \pm 1$. In particular, $T(p, q)$ can be considered to be a twisted torus knot of the type $K(p,q,1, \pm1)$, which is primitive/primitive and has surface slope $pq\pm1$.

Consider the torus knot $K_1 = T(kq+1, q) = K(kq+1, q, 1, -1)$, and the twisted torus knot $K_2 = K((k+1)q -1, q, q-1, -1)$ where $q \ge 3$ and $k \ge 2$. Each of these knots has surface slope $kq^2+q-1$. $K_1$ is primitive/primitive with respect to $F$ and has a lens space surgery at the surface slope. $K_2$ is of the form in Theorem \ref{psttks}, so $H[K_2] \cong D^2 (k, q-1)$, and $kq^2 +q+1$ surgery on $K_2$ yields a lens space which has the Seifert fibering $S^2(k, q-1, -1)$. These observations lead to the following result.

\begin{theorem}\label{p1}
$K_1$ and $K_2$ are isotopic as knots in $S^3$ and have the same surface slope with respect to $F$, but there is no homeomorphism of $S^3$ sending $(F, K_1)$ to $(F, K_2)$. 
\end{theorem}

\begin{proof}
To show the knots are isotopic in $S^3$, we express $K_1$ as the indicated surgery on the link shown in Figure
\ref{torusknot}, where the $q$ vertical strands will be closed to form a knot,
and $K_2$ as the indicated surgery on the link shown in Figure \ref{K2a}. In
each figure, a box with either 1 or -1 appears, as in Figure \ref{boxexp}(a).
Figure \ref{boxexp}(b) shows what is meant if a 1 appears in this box; in the
braid, the right-most strand passes over all of the other strands once, in
order. If a -1 appears, the braid will be a reflection of Figure
\ref{boxexp}(b), i.e. the left-most strand will pass under all of the other
strands once, in order. We will show that the knots shown in Figures \ref{torusknot} and \ref{K2a} are in fact the same. 

\pagebreak 
\begin{figure}[h]
	\centering
	\includegraphics{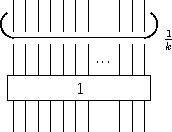}
	\caption{$T(p,q)$}\label{torusknot}
\end{figure}

\begin{figure}[ht]
\begin{center}
\includegraphics{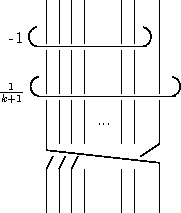}
\caption{$K((k+1)q-1, q, q-1, -1)$}\label{K2a}
\end{center}
\end{figure}

\begin{figure}[h]
 \begin{center}
\includegraphics[scale=.8]{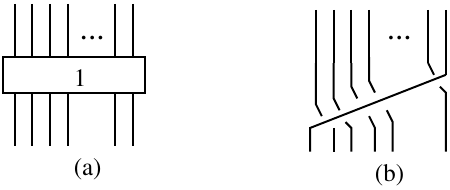} 
\caption{Twist box}\label{boxexp}
\end{center}
\end{figure}
\pagebreak

First reconsider the link for $K_2$ to be as shown in Figure \ref{K2b}, where
a component has been added, but the indicated surgeries will result in the
same knot. 

\begin{figure}[ht]
\begin{center}	
\includegraphics{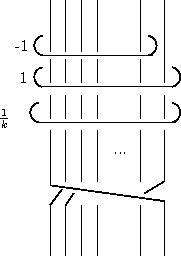}
\caption{$K((k+1)q-1, q, q-1, -1)$}\label{K2b}
\end{center}
\end{figure}

It is easy to see that the portion of $K_2$ shown in Figure \ref{K2part} can
be rewritten as the braid shown in Figure \ref{braidbits}, and  replacing that portion of $K_2$ appropriately, we obtain the link shown in
Figure \ref{K2c}, which is a reflection of $T(p,q)$, as shown in Figure
\ref{torusknot}.

\begin{figure}[ht]
 \begin{center}
\includegraphics[scale=.75]{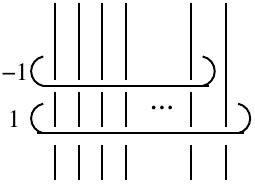} 
\caption{Part of $K_2$}\label{K2part}
\end{center}
\end{figure}

\begin{figure}[ht]
	\centering
	\includegraphics[scale=.7]{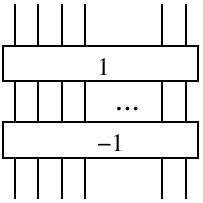}
	\caption{Part of $K_2$}\label{braidbits}
\end{figure}

\pagebreak

\begin{figure}[ht]
	\centering
	\includegraphics{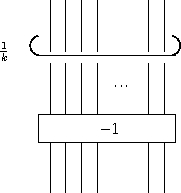}
	\caption{$K((k+1)q-1, q, q-1, -1)$}\label{K2c}
\end{figure}

Hence, as braids, the vertical strand portions of the links in Figures \ref{torusknot} and \ref{K2a} are conjugate by a half twist. As mentioned above, both of $K_1$ and $K_2$ have surface slope $kq^2 + q +1$. Then the indicated link surgeries represent the same knot, i.e. $K_1$  and $K_2$ are isotopic as knots in $S^3$. To show there is no homeomorphism of $S^3$ sending $(F, K_1)$ to $(F, K_2)$, we note that $K_2$ is $(k, q-1)$ Seifert with respect to $H$. Since $k\ge2$ and $q-1\ge2$, $H[K_2] \cong D^2(k, q-1)$. On the other hand, $K_1$ is primitive with respect to both $H$ and $H'$. Since a homeomorphism $h: S^3 \rightarrow S^3$ sending $(F, K_1)$ to $(F, K_2)$ will send $H$ to either $H$ or $H'$, no such homeomorphism can exist.
\begin{flushright}
$\square$
\end{flushright}
\end{proof}

\end{section}

\end{document}